\documentclass[letterpaper, 10 pt, conference]{ieeeconf}  % Comment this line out if you need a4paper

\IEEEoverridecommandlockouts                              % This command is only needed if 
\usepackage{cite}
\usepackage{amsmath,amssymb,amsfonts}
  \usepackage{amsthm}

\usepackage{algorithmic}
\usepackage{graphicx}
\usepackage{textcomp}
\usepackage{xcolor}
\usepackage{mathrsfs}
\usepackage{array}
\usepackage{multirow} 
\usepackage{subfigure}
\usepackage{graphicx}
\usepackage{algorithm}
\usepackage{algorithmic}
\usepackage{float}
\usepackage{color,xcolor}
\usepackage{array}
\usepackage{arydshln}
\usepackage{multirow} 
\usepackage{float}

\def \R {\mathbb{R}}

\def \N {\mathcal{N}}

\def \bb {{\bf b}}
\def \ba {{\bf a}}

\def \E {\mathbb{E}}
\def \O {\mathcal{O}}

\theoremstyle{plain}
\newtheorem{THEOREM}{Theorem}[section]

\newtheorem{theorem}[THEOREM]{Theorem}

\newtheorem{lemma}[THEOREM]{Lemma}

\theoremstyle{remark}
\newtheorem{example}[THEOREM]{Example}
\newtheorem*{assumption*}{Assumption}
\newtheorem*{remark*}{Remark}

\allowdisplaybreaks[4]

\title{\LARGE \bf
Error Estimates of the Gain Approximation by Hermite-Galerkin  Method in Feedback Particle Filter*
}

\author{Ruoyu Wang$^{1}$, {\it IEEE student member}, Peng Sun$^{1}$ and Xue Luo$^{1,2,\sharp}$, {\it IEEE senior member}% <-this % stops a space
\thanks{*This work is financially supported by National Natural Science Foundation of China (Grant No. 12271019) and the National Key R\&D Program of China (Grant No. 2022YFA1005103).}% <-this % stops a space
\thanks{$^{1}$ All the authors are with School of Mathematical Sciences, Beihang University, Beijing, P. R. China 102206. 
        {\tt\small WangRY@buaa.edu.cn, 1512530585@buaa.edu.cn}}
\thanks{$^{2}$ X. Luo is also with the Key Laboratory of Mathematics, Informatics and Behavioral Semantics (LMIB), Beihang University, Beijing, P. R. China 100191.
        {\tt\small xluo@buaa.edu.cn}}%
\thanks{$^{\sharp}$X. Luo is the corresponding author.}
}

\begin{document}

\maketitle
\thispagestyle{empty}
\pagestyle{empty}

%%%%%%%%%%%%%%%%%%%%%%%%%%%%%%%%%%%%%%%%%%%%%%%%%%%%%%%%%%%%%%%%%%%%%%%%%%%%%%%%
\begin{abstract}
The feedback particle filter (FPF) is a promising nonlinear filtering (NLF) method, but its practical implementation is hindered by the intractability of the gain function, which satisfies a boundary value problem (BVP). This paper proposes a novel two-step Hermite-Galerkin spectral method to address this challenge. First, the unknown density in the BVP is approximated by a kernel density estimator, whose error bounds are well-established in the literature. Second, rather than directly approximating the gain function, we approximate an auxiliary variable via the Galerkin spectral method using generalized Hermite functions. This auxiliary variable inherits the rapid decay property of the density at infinity, which aligns perfectly with the exponential decay characteristic of generalized Hermite functions, thereby obviating the need for artificial boundary conditions or domain truncation. Furthermore, we rigorously establish two fundamental error estimates: the kernel approximation error decays at the rate \(\O(N_p^{-\frac{s}{2s+1}})\), while the spectral approximation error converges at \(\O(M^{-s+1}\log M)\), providing complete theoretical guarantees for the method’s accuracy. Comprehensive numerical experiments validate the theoretical results and demonstrate that the proposed method outperforms existing gain approximation schemes in both accuracy and computational efficiency.
\end{abstract}

%%%%%%%%%%%%%%%%%%%%%%%%%%%%%%%%%%%%%%%%%%%%%%%%%%%%%%%%%%%%%%%%%%%%%%%%%%%%%%%%
\section{Introduction}\label{sec-1}

The nonlinear filtering (NLF) problem for diffusion processes is formulated as
\begin{equation}\label{eqn-NLF}
\left\{\begin{aligned}
dX_t &= g(X_t)dt+\sigma(X_t)dB_t,\\
dZ_t &= h(X_t)dt+dW_t,
\end{aligned}\right.
\end{equation}
where $X_t\in\mathbb{R}^d$ is the system state, $Z_t\in\mathbb{R}^m$ is the observation, and $\{B_t\}$, $\{W_t\}$ are mutually independent standard Wiener processes. The coefficients $g:\mathbb{R}^d\to\mathbb{R}^d$, $h:\mathbb{R}^d\to\mathbb{R}^m$, and $\sigma:\mathbb{R}^d\to\mathbb{R}^{d\times d}$ are continuous functions; $g$ and $\sigma$ are Lipschitz continuous.

Conventional approaches such as particle filters (PF) and extended Kalman filters either suffer from prohibitive computational complexity or experience severe performance degradation under strong nonlinearity. Introduced in 2013, the feedback particle filter (FPF) has emerged as an attractive alternative to conventional PF, owing to its intrinsic feedback control structure. In the FPF, each particle evolves according to a controlled stochastic differential equation
\begin{equation}\label{eqn-2.20}
dX_t^i = g(X_t^i)dt+\sigma(X_t^i)dB_t^i+K(X_t^i,t)dZ_t+u(X_t^i,t)dt,
\end{equation}
where $X_t^i$ denotes the $i$-th particle's state at time $t$. The gain function $K$ in \eqref{eqn-2.20} satisfies the Euler-Lagrange boundary value problem (BVP)
\begin{equation}\label{eqn-BVP}
\operatorname{div}(p_tK)=-\left(h-\hat{h}\right)^Tp_t,
\end{equation}
with the boundary condition $\displaystyle\lim_{|x|\to\infty}p_tK=0$, 
\begin{equation}\label{eqn-1.1}
    \hat{h}=\int_{\mathbb{R}^d}h(x)p_t(x)dx,
    \end{equation}
    and
\begin{equation}\label{eqn-u}
    u=-\frac12K\left(h+\hat h\right)+\frac12\sum_{k=1}^d\sum_{s=1}^mK_{ks}\frac{\partial K_{ls}}{\partial x_k}.
\end{equation}
As established in \cite{YMM:13}, the optimal gain $K$, the solution to \eqref{eqn-BVP}, ensures that the empirical distribution of particles coincides with the true conditional state density $p_t$. Surace et al. \cite{SKP:19} further validated numerically that the optimal gain function allows FPF to outperform nearly all conventional NLF methods.
Nonetheless, as highlighted in \cite{TM:16}, solving \eqref{eqn-BVP} remains a formidable challenge, for two key reasons:
\begin{enumerate}
    \item[1)] In general, \eqref{eqn-BVP} does not admit a closed-form solution, even when $p_t$ is assumed known. The true density $p_t$ is precisely the target of the NLF problem and thus is inherently unavailable.
    \item[2)] Accordingly, one must first construct a reliable approximation of $p_t$, and then accurately and efficiently solve \eqref{eqn-BVP} using this approximate density.
\end{enumerate}

Several efforts have been devoted along this line. As proposed in \cite{YMM:13}, the true density \(p_t\) can be approximated by its empirical counterpart \(p_t^{(N_p)}\), whose almost-sure convergence to \(p_t\) is guaranteed by the law of large numbers. By substituting \(p_t^{(N_p)}\) for \(p_t\) in \eqref{eqn-BVP}, various numerical gain approximations have been developed. The constant-gain approximation \cite{YMM:13,YLMM:16} assumes a uniform gain across all particles.
A Galerkin-type scheme was first introduced in \cite{YLMM:13}, with polynomial basis functions adopted in \cite{TM:16}, which suffers from limited scalability and potential Gibbs oscillations in high dimensions.
Berntorp et al. \cite{BG:16} developed a proper orthogonal decomposition approach to bypass prescribed basis functions. Taghvaei et al. \cite{TM:16,TMM:20} presented kernel-based methods that avoid explicit basis construction entirely. A comprehensive comparison of these gain approximation techniques for FPF was provided in \cite{B:18}. In our recent work \cite{WML:25}, we derived an exact gain function for FPF under polynomial observation models by decomposing the BVP \eqref{eqn-BVP} into two exactly solvable components.
This framework was further extended to multivariate NLF problems with polynomial observations in \cite{WL:25}.

Over the past decades, Galerkin spectral methods have stood out in solving NLF problems owing to their exponential convergence properties.
The last author and her collaborator \cite{LY:13HSM} first introduced an on- and off-line algorithm based on the Hermite spectral method to solve the Duncan-Mortensen-Zakai equation for the conditional state density. Later, Legendre-Galerkin schemes on bounded domains were investigated in \cite{DLY:21}. Furthermore, NLF problems with correlated noises have been addressed via Galerkin spectral methods in \cite{ZZCZC:22,SY:23}. Most recently, deep learning‑augmented spectral approximations were studied in \cite{SJY:25}. These studies consistently demonstrate the superior performance of spectral methods in approximating conditional densities.

In this paper, to address the two aforementioned challenges in solving the BVP \eqref{eqn-BVP} for the gain function, we propose a novel approach consisting of the following two key steps:
\begin{enumerate}
    \item[1)] Approximate the unknown density \(p_t\) using the kernel density estimator \(p_t^{N_p,\epsilon}\) (see \eqref{kernel}), whose error bounds have been well-documented in the nonparametric statistics literature.
    \item[2)] Approximate the auxiliary variable \(p_t^{N_p,\epsilon}K\) (on the left-hand side of \eqref{eqn-BVP}) as an integral quantity via the Galerkin spectral method, employing generalized Hermite functions (see \eqref{eqn-2.1}).
\end{enumerate}
A critical distinction between our Hermite-Galerkin method and existing Galerkin-type approaches in \cite{YLMM:13,TM:16} lies in the choice of the auxiliary variable \(p_t^{N_p,\epsilon}K\) to be approximated. This choice is not only theoretically justified (see Theorem \ref{thm-3.2}) but also computationally advantageous, as demonstrated numerically in Section \ref{sec-4}. Specifically, the auxiliary variable inherits the rapid decay property of the kernel density estimator \(p_t^{N_p,\epsilon}\) at infinity, which aligns perfectly with the exponential decay characteristic of generalized Hermite functions. This alignment obviates the need for artificial boundary conditions or domain truncation. Furthermore, the kernel density estimator is more suitable than the empirical distribution $p_t^{(N_p)}$, as the error estimate in Theorem \ref{thm-3.2} depends on the regularity of \(f_{N_p}\), the solution to \eqref{eqn-BVP-F}.

The main contributions of this paper are twofold: 1) The development of a Hermite–Galerkin spectral method for the numerical solution of the gain function in the FPF, which is governed by the BVP \eqref{eqn-BVP}. 2) The derivation of rigorous error bounds for the two successive approximations employed-namely, kernel density estimation of the conditional density and Galerkin approximation of the auxiliary variable-thereby providing solid theoretical guarantees for the approximation accuracy.
Comprehensive numerical experiments are carried out to verify the theoretical error estimates, and the performance of the FPF equipped with the proposed gain approximation is systematically compared with other representative schemes on a benchmark NLF problem.

The organization of this paper is as follows: Section \ref{sec-2} reviews the definitions, key properties of generalized Hermite functions, and the associated Sobolev spaces, while also stating lemmas related to the existing error estimates of kernel density estimators. Section \ref{sec-3} develops the Hermite-Galerkin method and establishes two error estimates (Theorem \ref{thm-3.1} and Theorem \ref{thm-3.2}) to rigorously characterize the approximation accuracy. Section \ref{sec-4} presents several numerical experiments to verify the theoretical results, illustrate the effectiveness of the proposed method, and compare its performance with existing approaches. Finally, concluding remarks and future research directions are summarized in Section \ref{sec-5}.

\section{Preliminary}\label{sec-2}

In this section, we shall recall some definitions and well-known results.

\subsection{Generalized Hermite functions and Sobolev Space}

The generalized Hermite functions \cite{STW:11} are defined as
\begin{equation}\label{eqn-2.1}
	\tilde{H}_{n}(x)=\frac{1}{{\pi}^{1/4}\sqrt{2^{n}n!}}e^{-x^{2}/2}{H}_{n}(x), 
\end{equation}
for \(n=0,1,\cdots\) and \(x\in\mathbb{R}\), where \(H_n(x)\) denotes the Hermite polynomial of degree \(n\). These functions are mutually orthogonal, i.e.,
\begin{equation*}
    \int_{-\infty}^{+\infty}\tilde{H}_m(x)\tilde{H}_n(x)dx=\delta_{mn},
\end{equation*}
where \(\delta_{mn}\) is the Kronecker delta. Moreover, they inherit the three-term recursion relation from Hermite polynomials:
\begin{align}\label{hermite_functions}\notag
     &\tilde{H}_{0}(x)=\pi^{-\frac{1}{4}}e^{-x^2/2}, \quad \tilde{H}_{1}(x)=\sqrt{2}\pi^{-\frac{1}{4}}xe^{-x^2/2},\\
     &\tilde{H}_{n+1}(x)=\sqrt{\frac{2}{n+1}}x\tilde{H}_{n}(x)-\sqrt{\frac{n}{n+1}}\tilde{H}_{n-1}(x),   
\end{align}
for \(n=1,2,\cdots\). The derivative of \(\tilde{H}_n(x)\) satisfies
\begin{align}\label{HFD}\notag
    \tilde{H}'_{n}(x)&=\sqrt{2n}\tilde{H}_{n-1}(x)-x\tilde{H}_{n}(x)\\
    &\overset{\eqref{hermite_functions}}{=}\sqrt{\frac{n}{2}}\tilde{H}_{n-1}(x)-\sqrt{\frac{n+1}{2}}\tilde{H}_{n+1}(x).
\end{align}

The set of generalized Hermite functions \(\{\tilde{H}_n\}_{n\geq0}\) forms a complete orthonormal basis for \(L^2(\mathbb{R})\), providing a natural framework for analyzing function regularity and constructing finite-dimensional approximation subspaces. For any function \(f\in L^2(\mathbb{R})\), its Hermite expansion is given by
\begin{equation}\label{eqn-2.3}
    f(x) = \sum_{m=0}^{\infty}\hat{f}_{m} \tilde{H}_{m}(x),
\end{equation}
with expansion coefficients \(\hat{f}_{m} := \int_{\mathbb{R}} f(x)\tilde{H}_{m}(x)dx\).

For \(r\ge 0\), we define the Sobolev space
\begin{equation}\label{eqn-2.2}
    H^r(\mathbb{R}) := \Big\{ f\in L^2(\mathbb{R}) : \|f\|_r^2 := \sum_{m=0}^{\infty} (m+1)^r |\hat{f}_{m}|^2 < \infty \Big\},
\end{equation}
where the index \(r\) characterizes the regularity of \(f\). For integer \(r\), the norm \(\|f\|_r\) is equivalent to the standard Sobolev norm \(\|f\|_{H^r}\).

Let \(R_M = \textup{span}\{\tilde{H}_0,\cdots,\tilde{H}_M\}\) denote the linear subspace spanned by the first \(M+1\) generalized Hermite functions. The \(L^2\)-orthogonal projection \(P_M: L^2(\mathbb{R})\to R_M\) is expressed as
\begin{equation}\label{eqn-2.4}
    P_M f(x) = \sum_{m=0}^M \hat{f}_{m} \tilde{H}_{m}(x). 
\end{equation}

The following lemma quantifies the projection error on \(R_M\) in the Sobolev norm, which depends on \(M\) and the regularity of \(f\).
\begin{lemma}[Projection Error]\label{Projection error}
    Let \(f\in H^r(\mathbb{R})\) with \(r\ge 0\) and \(0\le \mu \le r\). Then 
    \begin{equation}\label{f-P_N f}
        \|f - P_M f\|_{H^\mu} \le CM^{\frac{\mu-r}{2}} \|f\|_r, 
    \end{equation}
   for some constant \(C=C(\mu,r)>0\) that is independent of \(M\).
\end{lemma}
The proof of this lemma can be found in Theorem 2.1 of \cite{LY:13SIAM}.

\subsection{Some error estimates of kernel density estimators}

As discussed in Section \ref{sec-1}, the true probability density \(p(x)\) generally does not admit an analytical expression in NLF problems. To obtain an approximate solution to \eqref{eqn-BVP}, a suitable approximation to \(p(x)\) must be employed. In this paper, we adopt the kernel density estimator for this purpose. Let \(X^1,\cdots,X^{N_p}\) be i.i.d.\ samples drawn from \(p(x)\); the kernel density estimator with bandwidth \(\epsilon>0\) is then defined as
\begin{equation}\label{kernel}
    p^{N_p,\epsilon}(x) := \frac{1}{N_p}\sum_{i=1}^{N_p} K_\epsilon(x - X^i),
\end{equation}
where \(K_\epsilon(u) := \epsilon^{-1}K(u/\epsilon)\). All error estimates derived in this paper hold for any kernel satisfying the stated conditions, while a Gaussian kernel is used in the numerical experiments of Section \ref{sec-4}.

In the nonparametric statistics literature, rigorous error bounds have been established for both the mean integrated squared error (MISE) and the expected \(L^1\)-error. Below, we summarize these results in two lemmas, and slightly extend the \(L^1\)-error bound to its squared version, which is essential for the error analysis in Section \ref{sec-3.2}.

The following lemma provides the MISE bound for the kernel estimator.

\begin{lemma}[Theorem 1.3, \cite{T:09}]\label{error of kernel density estimation}
    Let \(p\in H^s(\mathbb{R})\) be a probability density such that \(p^{(s)}(x)\in L^2(\mathbb{R})\). Suppose \(K_\epsilon:\mathbb{R}\to\mathbb{R}\) is a kernel of order \(l=\lfloor s-1\rfloor\) with bandwidth \(\epsilon>0\), satisfying
    \begin{equation}\label{eqn-2.5}
        \int_\mathbb{R} K_\epsilon(u)\,du=1,\quad
        \int_\mathbb{R} u^j K_\epsilon(u)\,du=0,\quad j=1,2,\cdots,l,
    \end{equation}
    together with \(\int_\mathbb{R} K_\epsilon^2(u)\,du<\infty\) and \(\int_\mathbb{R}|u|^s K_\epsilon(u)\,du<\infty\). Then for all \(N_p\ge1\) and \(\epsilon>0\), the MISE of \(p^{N_p,\epsilon}\) in \eqref{kernel} satisfies
    \begin{align*}
        \textup{MISE}:=&\mathbb{E}\|p^{N_p,\epsilon}-p\|_{L^2}^2\\
        \le&\,\frac1{N_p\epsilon}\int_\mathbb{R} K_\epsilon^2(u)\,du\\
        &+\frac{\epsilon^{2s}}{(l!)^2}\left(\int_\mathbb{R}|u|^s|K_\epsilon(u)|\,du\right)^2
        \|p^{(s)}\|_{L^2(\mathbb{R})}^2.
    \end{align*}
    In particular, with the optimal bandwidth
    \begin{equation}\label{eq:optimal-bandwidth}
        \epsilon_{N_p}^* = O\left(N_p^{-\frac{1}{2s+1}}\right),
    \end{equation}
    we have
    \[
        \textup{MISE} \le C N_p^{-\frac{2s}{2s+1}},
    \]
    for some constant \(C=C\big(s,K_{\epsilon^*_{N_p}},\|p^{(s)}\|_{L^2(\mathbb{R})}\big)>0\).
\end{lemma}

The expected squared \(L^1\)-error is also needed in the proof of Theorem \ref{thm-3.1}. While \(\mathbb{E}\|p-p^{N_p,\epsilon}\|_{L^1(\mathbb{R})}\) was analyzed in Theorem 12 of \cite{DG:85}, we provide a slight extension to \(\mathbb{E}\|p-p^{N_p,\epsilon}\|_{L^1(\mathbb{R})}^2\) in the next lemma.

\begin{lemma}[Squared \(L^1\)-error bound]\label{L1 Error Bound for Kernel Density Estimation}
Let \(p \in H^s(\mathbb{R})\) be a probability density satisfying \(\int_{\mathbb{R}} \sqrt{p(x)}\,dx < \infty\) and
\(C_s := \liminf_{a\downarrow 0} \|(p * \phi_a)^{(s)}\|<\infty\), where \(\phi_a\) denotes a standard mollifier. Let the kernel \(K_\epsilon\) satisfy the same conditions as in Lemma \ref{error of kernel density estimation}. Assume further that the bandwidth \(\epsilon=\epsilon_{N_p}\to0\) such that \(N_p \epsilon_{N_p}\to\infty\), as $N_p\to\infty$. Then
\begin{equation}\label{eq:l1-square-bound}
    \mathbb{E}\|p^{N_p,\epsilon_{N_p}} - p\|_{L^1(\mathbb{R})}^2
    \le \left(C_1 \epsilon_{N_p}^{2s} + \frac{C_2}{N_p \epsilon_{N_p}}\right)(1+o(1)),
\end{equation}
where $C_1 = 2C_s^2 \left(\int_{\mathbb{R}}|L(z)|\,dz\right)^2$, $
L(z) = \frac{(-1)^s}{(s-1)!}\int_z^\infty (y-z)^{s-1}K(y)\,dy$ and $
C_2 = 2\|K\|_{L^2(\mathbb{R})}^2 \left(\int_{\mathbb{R}}\sqrt{p(z)}\,dz\right)^2$. 

In particular, with the optimal bandwidth \(\epsilon_{N_p}^*\) in \eqref{eq:optimal-bandwidth}, the estimator achieves the optimal convergence rate
\begin{equation}\label{eq:optimal-rate}
\mathbb{E}\|p^{N_p,\epsilon_{N_p}^*} - p\|_{L^1(\mathbb{R})}^2 = O\left(N_p^{-\frac{2s}{2s+1}}\right).
\end{equation}
\end{lemma}

To streamline the presentation, a sketch of the proof is deferred to the appendix.

\section{Hermite-Galerkin approximation and error estimates\label{sec-3}}

In this section, we shall detail the Hermite-Galerkin method for \eqref{eqn-BVP}. As introduced in Section \ref{sec-1}, two approximations have been made:
\begin{enumerate}
    \item The unknown probability density $p_t$ is approximated by the kernel density estimator \eqref{kernel} and 
\begin{equation}\label{eqn-hat hN}
    \hat{h}^{N_p} :=\int_{\R}h(y)p^{N_p,\epsilon}(y)dy.
\end{equation}
    \item The Galerkin spectral method with generalized Hermite function \eqref{eqn-2.1} is used to obtain the approximate solution, denoted as $f_{N_p,M}$,
    to 
    \begin{equation}\label{eqn-BVP-F}
    f_{N_p}'(x)=-\left(h-\hat{h}^{N_p}\right)p^{N_p,\epsilon},
\end{equation}
where $M$ is the truncation in the spectral method.    
\end{enumerate}

 Since the time $t$ is fixed when solving \eqref{eqn-BVP}, we shall omit this subscription in the sequel. In Section \ref{sec-3.1}, we shall first formulate the Hermite-Galerkin spectral method for \eqref{eqn-BVP-F}. Then the error estimates of this method has been analyzed in Section \ref{sec-3.2} with respect to the number of the particles $N_p$ and the truncation $M$, respectively.

\subsection{Hermite-Galerkin method for \eqref{eqn-BVP-F}}\label{sec-3.1}

The weak formulation of \eqref{eqn-BVP-F} is 
\begin{equation}\label{eqn-weak-F}
\int_{\mathbb{R}} f'(x)\varphi(x)dx = -\int_{\mathbb{R}} (h(x)-\hat{h}^{N_p})p^{N_p,\epsilon}(x)\varphi(x)dx,
\end{equation}
for any test function $\varphi\in L_2(\mathbb{R})$. The Galerkin method is to find an approximate solution 
\begin{equation}\label{F_M}
   f_{N_p,M}(x)=\sum_{m=0}^M a_{m} \tilde{H}_{m}(x)\in R_M
\end{equation}
such that \eqref{eqn-weak-F} holds.

By choosing $\varphi(x)=H_l(x)$, and substituting \eqref{F_M} into \eqref{eqn-weak-F}, one has
\begin{align}\label{hermite-appro-total}\notag
    &-\int_{\mathbb{R}} (h(x)-\hat{h})p^{N_p,\epsilon}(x)\tilde{H}_{l}(x)dx\\\notag
    =&\sum_{m=0}^M a_{m} \int_{\mathbb{R}} \tilde{H}'_{m}(x)\tilde{H}_{l}(x)dx \\\notag
    \overset{\eqref{HFD}}=&\sum_{m=0}^M a_{m}\int_{\R}\left(\sqrt{\frac m2}\tilde H_{m-1}(x)-\sqrt{\frac{m+1}2}\tilde H_{m+1}(x)\right)\\\notag
    &\phantom{aaaaaaaaa}\cdot\tilde H_l(x)dx \\
    =&a_{l+1}\sqrt{\frac{l+1}2}-a_{l-1}\sqrt{\frac l2},
\end{align}
for $l=0,\cdots,M+1$, where $a_{-1}\equiv0$ by convention. Let us denote $\ba =(a_0,\cdots,a_M)$, if rewritten \eqref{hermite-appro-total} into matrix form, i.e. $A\ba=\bb$, where $A$ is a tri-diagonal matrix with $A_{l,l-1}=-\sqrt{\frac l2}$ and $A_{l,l+1}=\sqrt{\frac{l+1}2}$, and 
\begin{equation}\label{eqn-3.9}
    \bb_l:=-\int_{\mathbb{R}} (h(x)-\hat{h}^{N_p})p^{N_p,\epsilon}(x)\tilde{H}_{l}(x)dx,
\end{equation}
$l=0,\cdots,M+1$. Therefore, if $A$ is invertible, then $\ba=A^{-1}\bb$. Consequently, the gain function is 
\begin{equation}\label{K}
    K(x)\approx\frac{f_{N_p,M}(x)}{p^{N_p,\epsilon}(x)}
    \overset{\eqref{F_M}}=\frac{\sum_{m=0}^M a_{m} \tilde{H}_{m}(x)}{p^{N_p,\epsilon}(x)}.
\end{equation}

%$A\in\mathbb{R}^{(M+1)\times(M+1)}$ and the right‑hand side vector $b(t)\in\mathbb{R}^{M+1}$ by
%\begin{equation}
    %A_{mj} = -\int_{\mathbb{R}} \tilde{H}_j(x)\tilde{H}_m'(x)\,dx, 
%\end{equation}
%\begin{equation}
%b_m = \int_{\mathbb{R}} (h(x)-\hat{h})p_t\tilde{H}_m(x)\,dx.
%\end{equation}
%Using \eqref{HFD} and orthonormality we obtain
%\begin{equation}
%A_{mj} = 
%\left\{\begin{aligned}
%&-\sqrt{\frac{j}{2}}, m=j-1,\\
%&\sqrt{\frac{j+1}{2}}, m=j+1,\\
%&0, \text{otherwise}.
%\end{aligned}\right.,
%\end{equation}
%i.e., $A$ is a skew‑symmetric tridiagonal matrix independent of $t$. 

%We make the following assumptions:
%\begin{enumerate}
    %\item The true density satisfies $p_t \in H^s(\mathbb{R})$ with $s > 1/2$, i.e.,  
   %\begin{equation}
       %\int_{\mathbb{R}} |\partial_x^s p(x)|^2 \,dx < \infty.
   %\end{equation}
   %This condition ensures that, under an optimal bandwidth, the kernel density estimator achieves the $L^2$ convergence rate $n^{-s/(2s+1)}$, which is minimax optimal.
   %\item The function $h$ satisfies $h \in L^\infty(\mathbb{R}) \cap L^2(\mathbb{R})$ and is smooth. This assumption guarantees the boundedness of the operator $\mathcal{L}$ on $L^2(\mathbb{R})$; specifically, there exists a constant $C_\mathcal{L}>0$ such that  
   %$$
   %\|\mathcal{L}p_t\|_{L^2} \le C_\mathcal{L} %\|p_t\|_{L^2}, \forall p_t \in L^2.
   %$$  
   %\item The function $F = \mathcal{L}p$ belongs to the weighted Sobolev space $W^r(\mathbb{R})$, whose norm is defined by  
   %\begin{equation}
      % \|F\|_r^2 = \sum_{k=0}^\infty (k+1)^r |\hat{F}_k|^2,
   %\end{equation}
%\end{enumerate}

\subsection{Error estimates}\label{sec-3.2}

In this subsection, we shall analyze the errors introduced by the two steps mentioned at the beginning of Section \ref{sec-3}. In 1), the error is from the approximation of kernel density estimator. In 2), the error is from the truncation in the spectral method. 

\subsubsection{The error from kernel density estimator}

Let $f(x)=p(x)K(x)$ and $f_{N_p}(x)$ be the solution to \eqref{eqn-BVP} and  \eqref{eqn-BVP-F}, respectively.

\begin{theorem}\label{thm-3.1}
Assume that $p\in H^{s}(\mathbb{R})$, $s\geq2$, and $p^{N_p,\epsilon}$ is the kernel density estimator \eqref{kernel} based on $N_p$ i.i.d. particles $\{X^1,\cdots,X^{N_p}\}$ sampled from $p$ with the kernel $K_{\epsilon_{N_p}^*}$ satisfying the conditions in Lemma \ref{error of kernel density estimation} with the optimal bandwidth in \eqref{eq:optimal-bandwidth}. Suppose further that $h\in L^{2}(\mathbb{R})$. Then 
\begin{equation*}
\mathbb{E}\big|f(x)-f_{N_p}(x)\big|
\leq CN_p^{-\frac{s}{2s+1}},
\end{equation*}
for some constant $C>0$ depends on $s$, $K$, $\|h\|_{L^{2}(\mathbb{R})}$, $|\hat{h}|$, $\|p^{(s)}\|_{L^{2}(\mathbb{R})}$, $\int_\R\sqrt{p(y)}dy,L$ and $\|K\|_{L^2(\R)}$.
\end{theorem}
\begin{proof}
Let us take the difference of $f(x)$ and $f_{N_p}(x)$:
\begin{align}\label{eqn-3.1}\notag
    &f(x) - f_{N_p}(x)\\\notag
    \overset{\eqref{eqn-BVP},\eqref{eqn-BVP-F}}=&\int_{-\infty}^x \left[ -(h-\hat{h}) p(y) + (h-\hat{h}^{N_p}) p^{N_p,\epsilon}(y) \right] dy \\\notag
    =& \int_{-\infty}^x -h \left[ p(y) - p^{N_p,\epsilon}(y) \right] dy \\\notag
    &+ \int_{-\infty}^x \left( \hat{h} p(y) - \hat{h}^{N_p} p_t^{N_p,\epsilon}(y) \right) dy\\\notag
    =&\int_{-\infty}^x -h \left[ p(y) - p^{N_p,\epsilon}(y) \right] dy
    +\int_{-\infty}^x(\hat h-\hat h^{N_p})p(y)dy\\
    &+\hat h^{N_p}\int_{-\infty}^x (p(y)-p^{N_p,\epsilon}(y))dy
\end{align}
under the boundary condition $\displaystyle\lim_{y\to-\infty} [f(y)-f_{N_p}(y)] = 0$. Thus, by H\"older's inequality, we have
\begin{align}\label{eqn-3.2}\notag
&\left| f(x) - f_{N_p}(x) \right|\\\notag
%\leq &\left| \int_{-\infty}^x h \left[ p_t(y) - p_t^{N_p,\epsilon}(y) \right] dy \right|\\
%+ & \left| \int_{-\infty}^x \left( \hat{h} p_t(y) - \hat{h}^{N_p} p_t^{N_p,\epsilon}(y) \right) dy \right| \\
%\leq &\|h\|_{L^2(\mathbb{R})} \left\| p_t - p_t^{N_p,\epsilon} \right\|_{L^2(\mathbb{R})} \\
%+& \left| \int_{-\infty}^x \left( \hat{h} - \hat{h}^{N_p} \right) p_t(y)
%+ \hat{h}^{N_p} \left( p_t(y) - p_t^{N_p,\epsilon}(y) \right) dy \right|.\\
\overset{\eqref{eqn-3.1}}\leq &\|h\|_{L^2(\mathbb{R})} \left\| p- p^{N_p,\epsilon} \right\|_{L^2(\mathbb{R})}\\\notag
&+ \int_{\mathbb{R}} \left| \hat{h} - \hat{h}^{N_p}\right|p(y) dy
+ \hat{h}^{N_p} \int_{\mathbb{R}} \left| p(y) - p^{N_p,\epsilon}(y)  \right| dy \\\notag
\leq &\|h\|_{L^2(\mathbb{R})} \left\| p- p^{N_p,\epsilon} \right\|_{L^2(\mathbb{R})}
+\left| \hat{h} - \hat{h}^{N_p} \right| \\\notag
&+|\hat{h}^{N_p}| \left\| p - p^{N_p,\epsilon} \right\|_{L^1(\mathbb{R})}\\\notag
\overset{\eqref{eqn-1.1},\eqref{eqn-hat hN}}=&\|h\|_{L^2(\mathbb{R})} \left\| p - p^{N_p,\epsilon} \right\|_{L^2(\mathbb{R})}
+ \left| \int_{\mathbb{R}} h \left( p - p^{N_p,\epsilon} \right) dy \right| \\
&+|\hat{h}^{N_p} |\left\| p - p^{N_p,\epsilon} \right\|_{L^1(\mathbb{R})}\\\notag
\leq &2\|h\|_{L^2(\mathbb{R})} \left\| p - p^{N_p,\epsilon} \right\|_{L^2(\mathbb{R})}
+ |\hat{h}^{N_p}| \left\| p - p^{N_p,\epsilon} \right\|_{L^1(\mathbb{R})}.
\end{align}
Taking expectation on the both sides of \eqref{eqn-3.2}, we have
\begin{align}\label{eqn-3.3}\notag
    &\E|f(x)-f_{N_p}(x)|\\\notag
    \leq& 2\|h\|_{L^2(\R)}\E\left\| p - p^{N_p,\epsilon} \right\|_{L^2(\mathbb{R})}
    +\E\left[|\hat{h}^{N_p}| \left\| p - p^{N_p,\epsilon} \right\|_{L^1(\mathbb{R})}\right]\\\notag
    \leq& 2\|h\|_{L^2(\R)}\E\left\| p - p^{N_p,\epsilon} \right\|_{L^2(\mathbb{R})} +|\hat h|\E\left\| p - p^{N_p,\epsilon} \right\|_{L^1(\mathbb{R})}\\
    &+\E\left[|\hat{h}^{N_p}-\hat h| \left\| p - p^{N_p,\epsilon} \right\|_{L^1(\mathbb{R})}\right].
\end{align}
The last term on the right-hand side of \eqref{eqn-3.3} can be controlled by 
\begin{align}\label{eqn-3.4}\notag
    &\E\left[|\hat{h}^{N_p}-\hat h| \left\| p - p^{N_p,\epsilon} \right\|_{L^1(\mathbb{R})}\right]\\
    \leq&\left[\E|\hat{h}^{N_p}-\hat h|^2\right]^{\frac12}\left[\E\left\| p - p^{N_p,\epsilon} \right\|^2_{L^1(\mathbb{R})}\right]^{\frac12}\\\notag
    \leq&\|h\|_{L^2(\R)}\left[\E\left\| p - p^{N_p,\epsilon} \right\|^2_{L^2(\mathbb{R})}\right]^{\frac12}\left[\E\left\| p - p^{N_p,\epsilon} \right\|^2_{L^1(\mathbb{R})}\right]^{\frac12},
\end{align}
by H\"older's inequality again and the last inequality follows from 
\begin{align*}
    \E|\hat{h}^{N_p}-\hat h|^2
    =&\E\left|\int_{\R}h(p(y)-p^{n_p,\epsilon}(y))dy\right|^2\\
    \leq&\|h\|_{L^2(\R)}^2\E\left\| p - p^{N_p,\epsilon} \right\|^2_{L^2(\mathbb{R})}.
\end{align*}
By Lemma \ref{error of kernel density estimation} and \ref{L1 Error Bound for Kernel Density Estimation}, we have
\begin{align}\label{eqn-3.5}\notag
    &\E\left\| p - p^{N_p,\epsilon} \right\|_{L^2(\mathbb{R})}\\\notag
    \leq&\left[\E\left\| p - p^{N_p,\epsilon} \right\|^2_{L^2(\mathbb{R})}\right]^{\frac12}\left(\int_\R p(y)dy\right)^{\frac12}\\
    \leq& \left[\E\left\| p - p^{N_p,\epsilon} \right\|^2_{L^2(\mathbb{R})}\right]^{\frac12}
    \leq CN_p^{-\frac s{2s+1}},
\end{align}
so does $\E\left\| p - p^{N_p,\epsilon} \right\|_{L^1(\mathbb{R})}\leq C N_p^{-\frac s{2s+1}}$. Therefore, by substituting \eqref{eqn-3.4} and \eqref{eqn-3.5} back into \eqref{eqn-3.3}, one has
\begin{align*}
    &\E|f(x)-f_{N_p}(x)|\\
    \leq&2\|h\|_{L^2(\R)}\E\left\| p - p^{N_p,\epsilon} \right\|_{L^2(\mathbb{R})} +|\hat h|\E\left\| p - p^{N_p,\epsilon} \right\|_{L^1(\mathbb{R})}\\
    &+\|h\|_{L^2(\R)}\left[\E\left\| p - p^{N_p,\epsilon} \right\|^2_{L^2(\mathbb{R})}\right]^{\frac12}\\
        &\phantom{aa}\cdot\left[\E\left\| p - p^{N_p,\epsilon} \right\|^2_{L^1(\mathbb{R})}\right]^{\frac12}
    \leq CN_p^{-\frac s{2s+1}},
\end{align*}
where $C>0$ depends on the quantities in Lemma \ref{error of kernel density estimation} and \ref{L1 Error Bound for Kernel Density Estimation}.
\end{proof}

\subsubsection{The error from the truncation}

In this subsection, our goal is to estimate the error  $\mathbb{E}\|f_{N_p,M}-f_{N_p}\|^2_{L^2(\R)}$, where $f_{N_p}$ and $f_{N_p,M}(x)\in R_M$ are the solution to \eqref{eqn-BVP-F} and its approximation using Hermite-Galerkin method, respectively. 

\begin{theorem}\label{thm-3.2}
Let $p\in H^s(\mathbb{R})$, $s\geq2$. The observation function $h(x)$ such that $\bb_l$ in \eqref{eqn-3.9} are well-defined, for all $l=0,\cdots,M+1$. Then 
\begin{equation}\label{eqn-3.18}
    \E\|f_{N_p,M} - f_{N_p}\|^2_{L^2(\R)} 
    \leq CM^{-s+1}\log M\E\|f_{N_p}\|_s^2,
\end{equation}
for some generic constant $C>0$ and $\log(\cdot)$ represents natural logarithm.
\end{theorem}
\begin{proof}
This error is divided into two parts: 
\begin{align}\label{eqn-3.8}
   &\E\|f_{N_p,M} - f_{N_p}\|^2_{L^2(\R)} \\\notag
    =& \E\|f_{N_p,M} - P_Mf_{N_p}\|^2_{L^2(\R)} + \E\|P_Mf_{N_p} - f_{N_p}\|^2_{L^2(\R)},
\end{align}
where $P_Mf_{N_p}$ is the projection of $f_{N_p}$ onto the subspace $R_M$ \eqref{eqn-2.4}. From Lemma \ref{Projection error}, the projection error is estimated by 
\begin{equation}\label{eqn-3.16}
     \E\|P_Mf_{N_p} - f_{N_p}\|^2_{L^2(\R)}
    \leq CM^{-s}\E\|f_{N_p}\|^2_{s},
\end{equation}
where $C=C(s)>0$, if $f_{N_p}\in H^s(\R)$. Therefore, in the sequel we only need to estimate the first term on the right-hand side of \eqref{eqn-3.8}. Due to the orthogonality of the generalized Hermite functions, one has 
\begin{equation}\label{eqn-3.11}
    \E\|f_{N_p,M} - P_Mf_{N_p}\|^2_{L^2(\R)}=\E\left(\sum_{m=0}^M|a_m-\hat f_{N_p,m}|^2\right).
\end{equation}
Thus, we shall investigate the coefficients of $f_{N_p,M}$ and $P_Mf_{N_p}$ below. On the one hand, from \eqref{hermite-appro-total}, the coefficients $a_m$ of $f_{N_p,M}$ satisfies 
\begin{equation}\label{eqn-3.6}
    \bb_m=a_{m+1}\sqrt{\frac{m+1}2}-a_{m-1}\sqrt{\frac m2},
\end{equation}
for $m=0,\cdots,M-1$, and 
\begin{equation}\label{eqn-3.7}
    \bb_M=-a_{M-1}\sqrt{\frac M2},\quad\bb_{M+1}=-a_M\sqrt{\frac{M+1}2},
\end{equation}
where $\bb_m$ is defined in \eqref{eqn-3.9}. On the other hand, the coefficients $\hat f_{N_p,m}$ of the solution $f_{N_p}$ in \eqref{eqn-2.3} to \eqref{eqn-BVP-F} also can be expressed in the form \eqref{eqn-3.6} without truncation, i.e. $m=0,1,\cdots$. Consequently, the coefficients $\hat f_{N_p,m}$ of $P_Mf_{N_p}(x)$ satisfies 
\begin{equation}\label{eqn-3.10}
    \bb_m=\hat f_{N_p,m+1}\sqrt{\frac{m+1}2}-\hat f_{N_p,m-1}\sqrt{\frac m2},
\end{equation}
$m=0,1,\cdots,M+1$. From \eqref{eqn-3.6}-\eqref{eqn-3.10}, it yields that
\begin{align}\label{eqn-3.12}
    &|\hat f_{N_p,M}-a_M|=\sqrt{\frac{M+2}{M+1}}|\hat f_{N_p,M+2}|,\\
   & |\hat f_{N_p,M-1}-a_{M-1}|=\sqrt{\frac{M+1}{M}}|\hat f_{N_p,M+1}|,
\end{align}
and
\begin{align}\label{eqn-3.13}\notag
   & |\hat f_{N_p,m}-a_m|=\sqrt{\frac{m+2}{m+1}}|\hat f_{N_p,m+2}-a_{m+2}|\\
        =&\left\{\begin{aligned}
        &\textup{either}\\
            &\phantom{aa}\sqrt{\frac{m+2}{m+1}}\sqrt{\frac{m+4}{m+3}}\cdots\sqrt{\frac{M-1}{M-2}}|\hat f_{N_p,M-1}-a_{M-1}|,\\
        &\textup{or}\\
          & \phantom{aa} \sqrt{\frac{m+2}{m+1}}\sqrt{\frac{m+4}{m+3}}\cdots\sqrt{\frac{M}{M-1}}|\hat f_{N_p,M}-a_{M}|.
            \end{aligned}\right.
\end{align}
Substituting \eqref{eqn-3.12}-\eqref{eqn-3.13} back to \eqref{eqn-3.11}, one has 
\begin{align}\label{eqn-3.14}\notag
    & \E\|f_{N_p,M} - P_Mf_{N_p}\|^2_{L^2(\R)}\\\notag
     \leq&M\sum_{m=0}^{M-2}\frac 1{m+1}\E|\hat f_{N_p,M+1}|^2+\frac{M+2}{M+1}\E|\hat f_{N_p,M+2}|^2\\
     \leq& CM\log M\E|\hat f_{N_p,M+1}|^2+2\E|\hat f_{N_p,M+2}|^2.
\end{align}
for some generic constant $C>0$, which may vary from line to line. If $f_{N_p}\in H^s(\R)$, then from \eqref{eqn-2.2}, 
\begin{equation}\label{eqn-3.15}
    |\hat f_{N_p,M+1}|^2\leq (M+2)^{-s}\|f_{N_p}\|_s^2,
\end{equation}
so does $|\hat f_{N_p,M+2}|^2$. Substituting \eqref{eqn-3.15} back to \eqref{eqn-3.14}, one has
\begin{equation}\label{eqn-3.17}
    \E\|f_{N_p,M} - P_Mf_{N_p}\|^2_{L^2(\R)}\leq CM^{-s+1}\log M\E\|f_{N_p}\|_s^2.
\end{equation}
Equation \eqref{eqn-3.18} follows immediately by substituting \eqref{eqn-3.16} and \eqref{eqn-3.17} back to \eqref{eqn-3.8}.
\end{proof}

\subsection{The feedback particle filter (FPF) algorithm}

In Section \ref{sec-3.1}, we introduced the Hermite-Galerkin method to approximation the gain function in the FPF, which is crucial. The whole procedure of FPF is first to discretize the total experimental time $[0,T]$ by step size $\Delta t$. The gain function is approximated at each time instant $t_k=k\Delta t$, $k=0,\cdots,\lfloor\frac T{\Delta t}\rfloor$. During each time interval $[t_k,t_{k+1})$, the gain function is assumed unchange with respect to time.

The FPF with gain function approximated by Hermite-Galerkin  method is detailed in Algorithm \ref{alg-1}.
\begin{algorithm}[h!]
\caption{The FPF with gain function approximated by Hermite-Galerkin method}
\begin{algorithmic}[1]\label{alg-1}
    \STATE \% Initialization
        \FOR{$i=1$ to $N_p$}
		  \STATE{Sample $X_0^i$ from $p_0(x)$}
		\ENDFOR
    \STATE \% The FPF
    \FOR{$k=0$ to $k=\lfloor T/\Delta t\rfloor$}
		\STATE Approximate $\hat{h}^{N_p}$ in \eqref{eqn-hat hN} by sample mean $\frac1{N_p}\sum_{i=1}^{N_p}h(X_{t_k}^i)$.
		\FOR{$i=1$ to $N_p$}
            \STATE Generate $N_p$ independent samples of $\Delta B^{i}_{t_k}$ from $\N(0,\Delta t)$
            \STATE \% The Hermite-Galerkin method for the gain function
            \STATE Calculate $K$ and $u$ for the $i$-th particle $X_{t_k}^{i}$ by \eqref{K} and \eqref{eqn-u}, respectively
            \STATE Evolve the particles $\{X_{t_k}^{i}\}_{i = 1}^{N_p}$ according to \eqref{eqn-2.20}, i.e.
            \begin{align*}
                X_{t_{k+1}}^{i}=&X_{t_k}^{i}+g(X_{t_k}^i)\Delta t+\sigma(X_{t_k}^i)\Delta B_{t_k}^i\\
                &+K(X_{t_k}^i,t)\Delta Z_{t_k}+u(X_{t_k}^i,t)\Delta t.
            \end{align*} 
        \ENDFOR
    \ENDFOR
\end{algorithmic}
\end{algorithm}

\section{Numerical simulations}\label{sec-4}

Three experiments are conducted: 1) Compare the Hermite-Galerkin approximated gain function with the exact one (when \(p(x)\) is known); 2) Numerically validate the theoretical error estimates in Section \ref{sec-3.2}; 3) Solve a NLF problem and compare the averaged root mean square error (ARMSE) \eqref{eqn-4.2} of the FPF with our method, constant gain, and kernel-based gain approximations. A Gaussian kernel is used for the kernel density estimator \eqref{kernel} throughout.

\subsection{The comparison of the gain functions}\label{sec-4.1}

For the scalar case, the exact gain function is obtainable via direct integration: 
\begin{equation}\label{eqn-4.1}
    K(x)=-\frac1{p(x)}\int_{-\infty}^x\left(h(y)-\hat h\right)p(y)dy,
\end{equation}
which serves as a benchmark for evaluating approximation accuracy.

\begin{example}[Section III, \cite{TM:16}]\label{ex-1}
    Suppose $p$ is a mixture of two Gaussian distributions, given by $\frac12\N(-\mu,\sigma^2)+\frac12\N(\mu,\sigma^2)$, where $\mu=1$ and $\sigma^2=0.2$. The observation is $h(x)=x$. Let the number of particles $N_p=200$ and the bandwidth $\epsilon=0.5$.
\end{example}

In Fig. \ref{fig-1}, we display the gain functions obtained from Hermite-Galerkin method for $M=1,4$ and $7$, respectively. As the spectral truncation order $M$ increases, the gain approximation gets closer to the exact one.

\begin{figure}[htbp!]
    \centering
    \includegraphics[width=0.5\textwidth, trim = 10 190 10 210,clip]{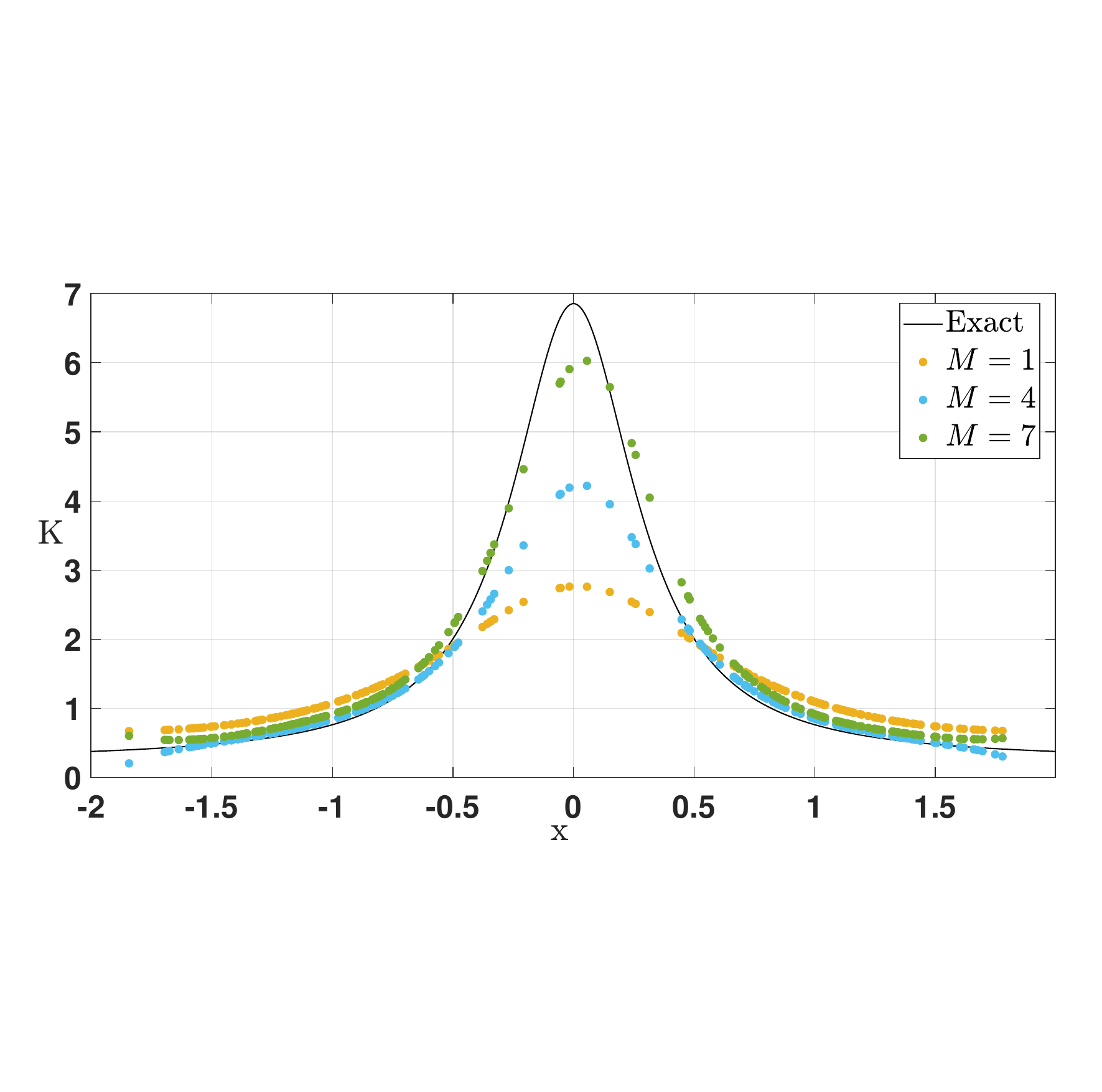}
    \caption{Comparison of the exact solution and its approximations by Hermite-Galerkin method with $M = 1,4$ and $7$.}\label{fig-1}
\end{figure}

\subsection{Validations of the theoretical error estimates}

Two tests validate the error estimates in Theorems \ref{thm-3.1} and \ref{thm-3.2}, using the bi-modal distribution from Example \ref{ex-1}. 
\begin{example}\label{ex-2}
   \begin{enumerate}
        \item[1)] Set the number of particles to be $N_p=200$. The $L^2$-errors between the true gain function and the approximated one obtained by the Hermite-Galerkin method, corresponding to different truncations $M=2,4,6,8$ and $10$, are recorded.
        \item[2)] Set the truncation $M=10$. The $L^1$-errors of those corresponding to varying number of particles $N_p=10, 30, 50, 100$, and $200$ are recorded.
   \end{enumerate}
\end{example}
Fig. \ref{fig-2} presents results as a log-log plot. The Gaussian kernel (order \(s=2\), \(\int_\mathbb{R} u^2K_\epsilon(u)du\neq0\)) aligns with Theorem \ref{thm-3.2}, with the theoretical \(L^2\)-error bound \(M^{-1}\log M\) (dashed line) matching the experimental error. The bottom subplot’s fitted slope (\(-0.4193\)) matches the theoretical rate \(N_p^{-0.4}\) (\(s=2\)).
\begin{figure}[h!]
    \centering
    \includegraphics[width=0.5\textwidth, trim = 0 190 10 205,clip]{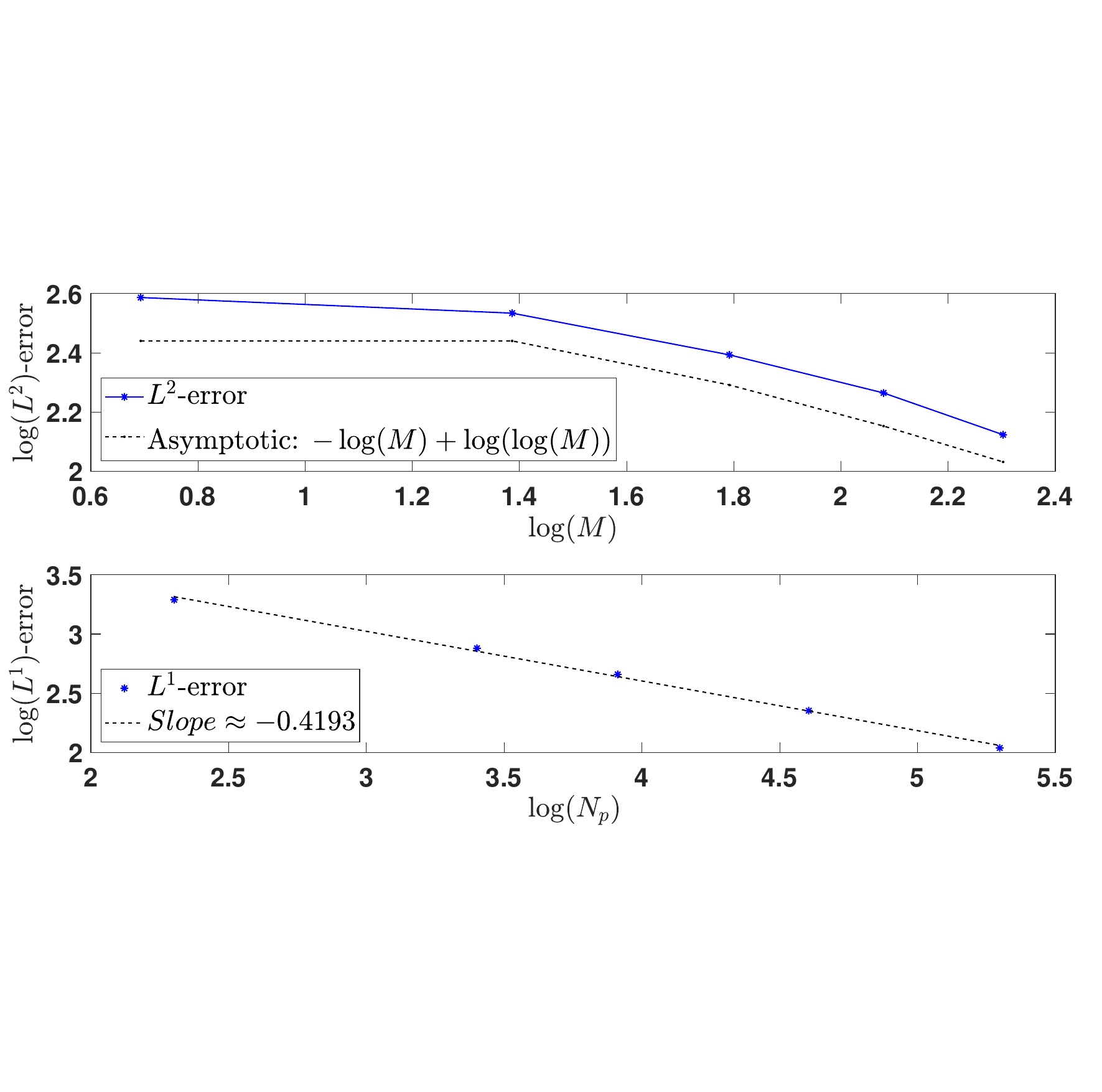}
    \caption{The log-log plot of the error estimates of the gain function approximated by the Hermite-Galerkin method. Top: the $\log(L^2)$-error v.s. $\log M$; Bottom: the $\log(L^1)$-error v.s. $\log N_p$.}\label{fig-2}
\end{figure}

\subsection{Nonlinear Example}\label{sec-4.3}

A NLF problem with state transitions between \(\pm1\) illustrates the algorithm's tracking capability.

 \begin{example}[Section V.B, \cite{YMM:13}]
\begin{equation}\label{NE1}
\left\{\begin{aligned}
dX_t=&X_t(1-X_t^2)dt+dB_t\\
dZ_t=&X_tdt+dW_t
\end{aligned}\right. .
\end{equation}
\end{example}

In this experiment, the total experimental time is $T=400$, with time step $\Delta t = 0.01$. The covariances of both the state and observation processes are $0.4$. Initially, the true state is around $X_0\approx0.1$. We generate $N_p=10$ particles from $\N(0,1)$. The realization of the true state is generated by the Euler-Maruyama method according to the first equation in \eqref{NE1}. The truncation is set to $M=6$ and the bandwidth is $\epsilon= 0.5$. From \eqref{eqn-3.6}-\eqref{eqn-3.7}, the coefficients $a_m$ in the Hermite-Galerkin approximation can be obtained backward recursively. Then the approximate gain function is obtained by \eqref{K}.

We illustrate the performance of FPF for one realization of (\ref{NE1}). The transition of true state between $\pm1$ can be obviously observed. Fig. \ref{fig-3} shows one realization of the state trajectory and its estimations from the FPF with three different gain approximations: the constant-gain approximation \cite{YLMM:13}, the kernel-based method \cite{TM:16,TMM:20}, and our Hermite-Galerkin method. Our method provides accurate tracking even at the sharp state transitions.

\begin{figure}[htbp!]
	\centering
	\includegraphics[width=0.5\textwidth, trim = 5 190 10 200,clip]{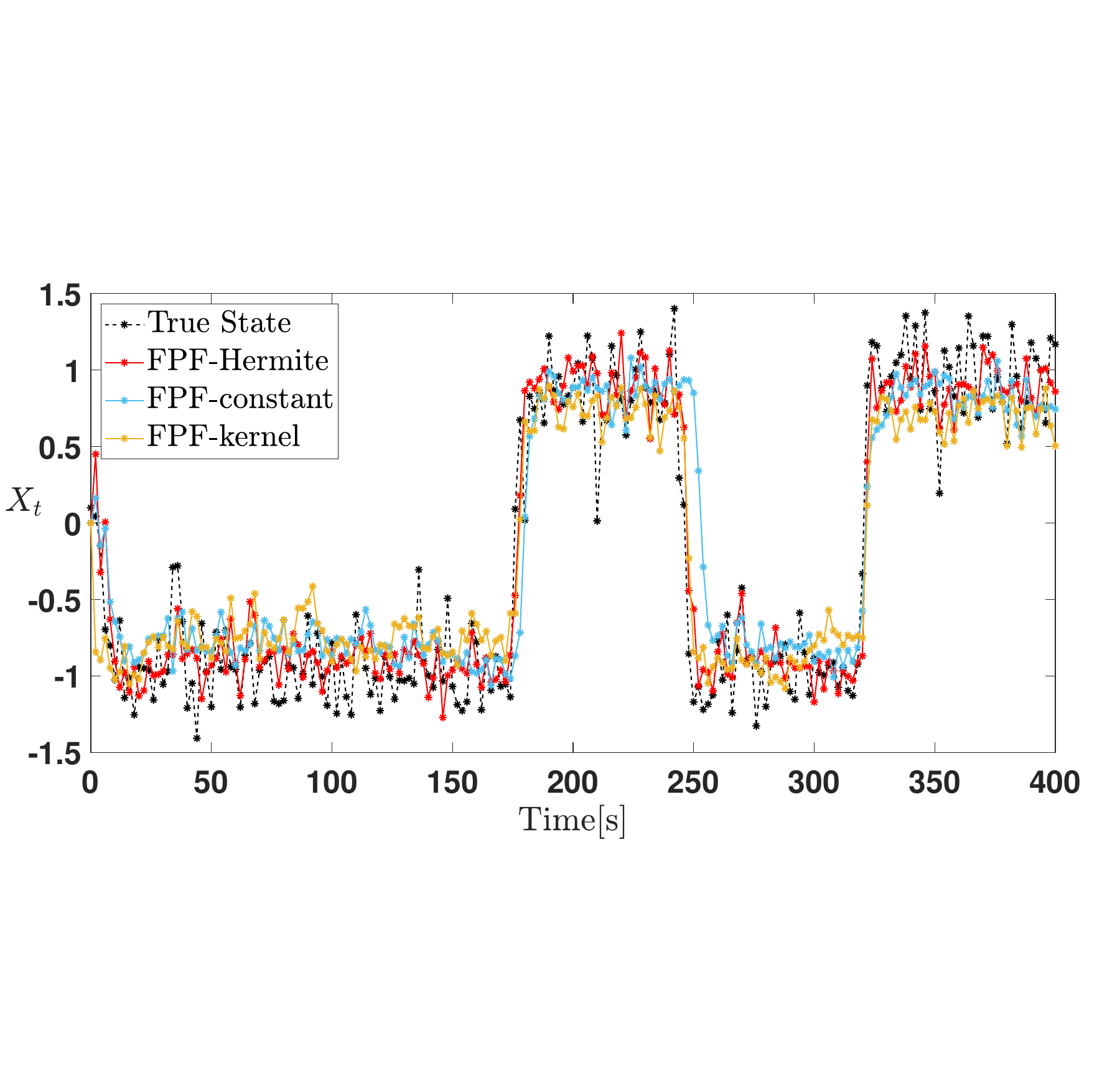}
	\caption{The estimations of the true state (black) obtained by the FPF with the  Hermite-Galerkin spectral method (red), the kernel-based approach (yellow) and the constant-gain approximation (blue), respectively.}
	\label{fig-3}
\end{figure}

 To assess the robustness of the algorithm, we perform $M_c=100$ Monte Carlo (MC) simulations, and the ARMSE are presented in Fig. \ref{fig-4}. The ARMSE is defined as 
 \begin{equation}\label{eqn-4.2}
    \textup{ARMSE}:=\frac1{M_c}\sum_{j=1}^{M_c}\textup{RMSE}_j,
 \end{equation}
 where RMSE$_j$ is the RMSE of the $j$-th MC run, and 
 the RMSE for a single run is defined as $\textup{RMSE}:=\left[\sum\limits_{k=0}^{\lfloor T/\Delta t\rfloor}(X_{t_k}-\hat X_{t_k})^2\right]^\frac{1}{2}$, where $X_{t_k}$ and $\hat X_{t_k}$ are the true state and the estimation at time $t_k$, respectively. 
 \begin{figure}[h!]
	\centering
	\includegraphics[width=0.5\textwidth, trim = 10 200 10 210,clip]{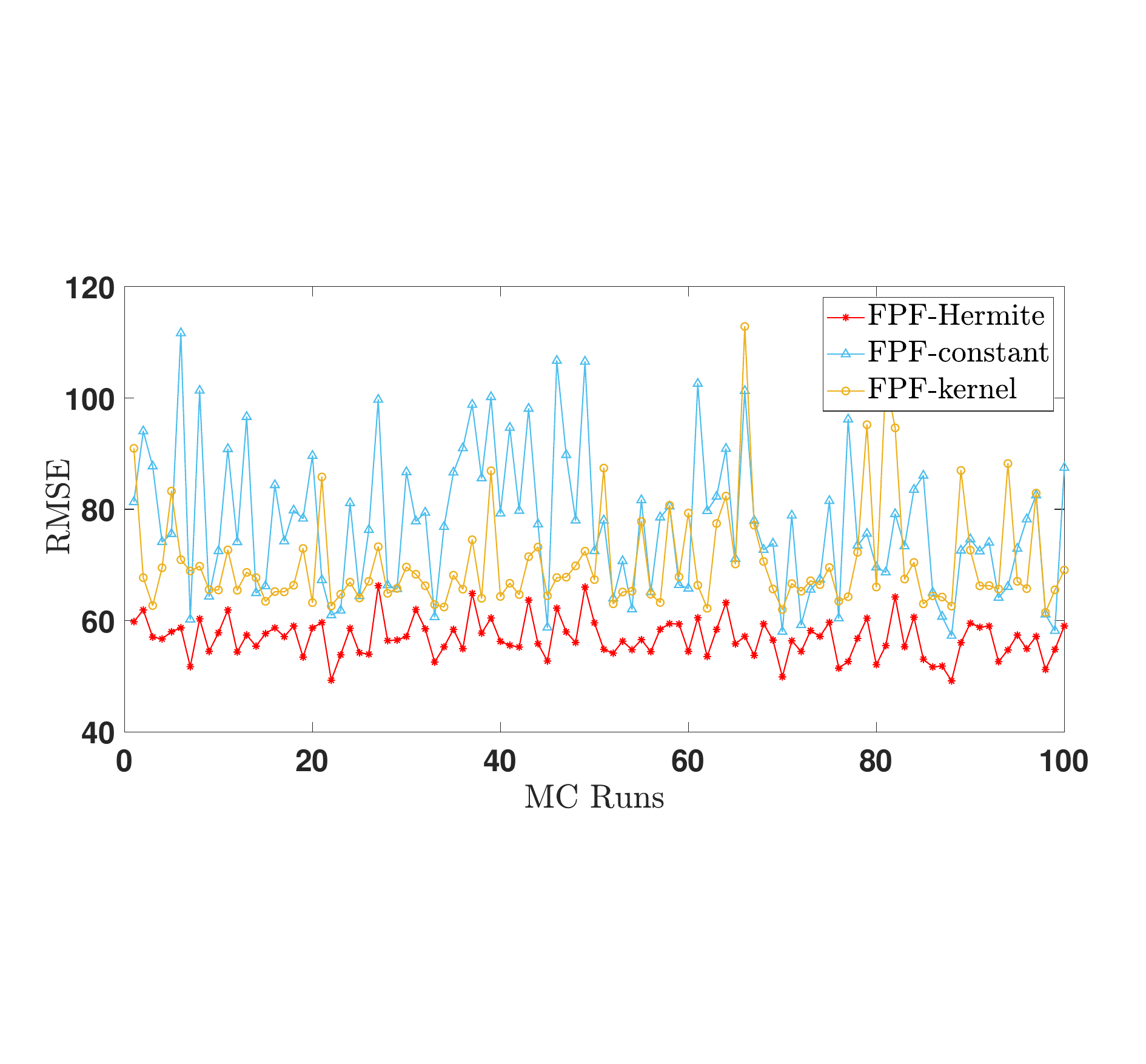}
	\caption{The RMSEs of each MC run of the three FPFs.}
	\label{fig-4}
\end{figure}

Table \ref{tab:comp} summarizes the ARMSEs and CPU times over $100$ trials of three FPF algorihtms.
 \begin{table}[h!]
    \centering
\begin{tabular}{|c|c|c|}
\hline
        & ARMSE & CPU times (s) \\
\hline
         Hermite-Galerkin spectral method & ${\bf 56.8836}$ & $3.0584$\\
    \hline
    kernel-based approach & $70.4475$ & $13.0531$\\
    \hline
    constant-gain approximation & $77.3933$ & ${\bf 0.1506}$\\
    \hline
\end{tabular}
\caption{Comparison of three FPFs in terms of ARMSE and CPU times.}
\label{tab:comp}
    \end{table}
The Hermite-Galerkin method achieves nearly $19\%$ and $26\%$ reductions in ARMSE compared to kernel-based approach and the constant-gain approximation, respectively. Regarding the computational efficiency, our method achieves a $77\%$ reduction in CPU times compared to the kernel-based approach, making it highly suitable for real-time NLF problems.

\section{Conclusions and future works}\label{sec-5}

In this paper, we have developed a novel Hermite-Galerkin spectral method to address the challenge of numerically solving for the gain function arising in the FPF. To overcome the two core difficulties (i.e. the intractability of the true conditional state density \(p_t\) and the lack of a closed-form solution to the BVP \eqref{eqn-BVP}), we proposed a two-step approximation framework: first, approximating \(p_t\) with a kernel density estimator; second, approximating the auxiliary variable \(p_t^{N_p,\epsilon}K\) via the Galerkin spectral method using generalized Hermite functions. 
We rigorously established two error bounds to quantify the approximation accuracy of the proposed method. Theorem \ref{thm-3.1} derived the pointwise error estimate for the kernel density approximation, showing that the expected $L^1$-error decays at the rate \(\O(N_p^{-\frac{s}{2s+1}})\) with respect to the number of particles \(N_p\), where \(s\) indicates the regularity of the true density. Theorem \ref{thm-3.2} further established the expected \(L^2\)-error bound for the spectral approximation, demonstrating that it decays at the rate \(\O(M^{-s+1}\log M)\) with respect to the number of basis functions \(M\), while inheriting the regularity of the solution to the BVP. These two error estimates provide complete theoretical guarantees for the convergence and accuracy of the proposed Hermite-Galerkin method.

Comprehensive numerical experiments were conducted to validate the theoretical error bounds and evaluate the practical performance of the method. The results confirmed that the proposed method achieves superior approximation accuracy and computational efficiency compared to existing gain function approximation schemes for the FPF, at least in the scalar and nonlinear scenarios.

Future research will focus on extending the current framework to multivariate NLF problems, exploring adaptive basis selection strategies to further improve computational scalability, and integrating the method with advanced sampling techniques to enhance the performance of FPF in real-world applications.

\bibliographystyle{ieeetr}
\bibliography{references}

@book{T:09,
    author = {Alexandre B. Tsybakov},
    title = {Introduction to Nonparametric Estimation},
    publisher = {Springer New York},
    year = {2009} 
}

@Conference{B:18,
	 title = {Comparison of gain function approximation methods in the feedback particle filter},
	author = {Berntorp, K.},
	booktitle = {2018 21th International Conference on Information Fusion},
	year = {2018},
	pages = {123--130},
}

@Conference{BG:16,
	 title = {Data-driven gain computation in the feedback particle filter},
	author = {Berntorp, K. and Grover, P.},
	booktitle = {Proceedings of the American Control Conference},
	year = {2016},
	pages = {2711--2716},
}

@Article{SKP:19,
  title =	{ How to avoid the curse of dimensionality: scalability of particle filters with and without importance weights},
  author =	 {Surace, S. and Kutschireiter, A. and Pfister, J.-P.},
  journal =	 {SIAM Review},
  volume =	 {61},
	number  = {1},
  year =	 2019,
  pages =	 {79-91},
}

@Conference{TM:16,
	 title = {Gain function approximation in the feedback particle filter},
	author = {Taghvaei, A. and Mehta, P.},
	booktitle = {Proceedings of 2016 IEEE Conference on Decision and Control},
	year = {2016},
	pages = {5446--5452},
}

@Article{TMM:20,
  title =	 {Diffusion map-based algorithm for gain function approximation in the feedback particle filter},
  author = {Taghvaei, A. and Mehta, P. and Meyn, S.},
 journal =	 {SIAM/ASA Journal on Uncertainty Quantification},
  volume =	 {8},
	number  = {3},
  year =	 2020,
 pages = {1090--1117},
}

@Article{WL:25,
  title =	 {A decomposition method in the multivariate feedback particle filter via tensor product Hermite polynomials},
  author = {Wang, R. and Luo, X.},
 journal =	 {arXiv:2511.01227v1},
  volume =	 {},
	number  = {},
  year =	 2025,
}

@ARTICLE{YMM:13,
  author={Yang, T. and Mehta, P. and Meyn, S.},
  journal={IEEE Transactions on Automatic Control}, 
  title={Feedback Particle Filter}, 
  year={2013},
  volume={58},
  number={10},
  pages={2465-2480},
  }

@Article{YLMM:16,
  title =	 {Multivariable feedback particle filter},
  author =	 {Yang, T. and Laugesen, R. and Mehta, P. and Meyn, S.},
  journal =	 {Automatica},
  volume =	 {71},
	number  = {},
  year =	 2016,
  pages =	 {10--23},
}

@INPROCEEDINGS{YLMM:13,
  author={Yang, T. and Laugesen, R. and Mehta, P. and Meyn, S.},
  booktitle={Proceedings of 2012 IEEE Conference on Decision and Control}, 
  title={Multivariable feedback particle filter}, 
  year={2012},
  volume={},
  number={},
  pages={4063-4070},
  }

@INPROCEEDINGS{WML:25,
      title={A Decomposition Approach for the Gain Function in the Feedback Particle Filter}, 
      author={R. Wang and H. Miao and X. Luo},
      booktitle = {Proceedings of 2025 IEEE Conference on Decision and Control},
      year={2025},
      pages={2378-2384}
}

@article{SY:23,
title = {Solving nonlinear filtering problems with correlated noise based on {H}ermite-{G}alerkin spectral method},
journal = {Automatica},
author = {Sun, Z. and Yau, S. S.-T.},
volume = {156},
year = {2023},
}

@article{LY:13HSM,
title = {Hermite spectral method to 1-D forward kolmogorov equation and its application to nonlinear filtering problems},
journal = {IEEE Transactions on Automatic Control},
author = {Luo, X. and Yau, S. S.-T.},
volume = {58},
number = {10},
year = {2013},
pages = {2495 - 2507},
}

@article{LY:13SIAM,
title = {Hermite spectral method with hyperbolic cross approximations to high-dimensional parabolic {PDE}s},
journal = {SIAM Journal on Numerical Analysis},
author = {Luo, X. and Yau, S. S.-T.},
volume = {51},
number = {6},
year = {2013},
pages = {3186 - 3212},
}

@article{DLY:21,
title = {Solving nonlinear filtering problems in real time by {L}egendre {G}alerkin spectral method},
journal = {IEEE Transactions on Automatic Control},
author = {Dong, W. and Luo, X. and Yau, S. S.-T.},
volume = {66},
number = {4},
year = {2021},
pages = {1559 - 1572},
}

@article{ZZCZC:22,
title = {Splitting-up Spectral Method for Nonlinear Filtering Problems with Correlation Noises},
journal = {Journal of Scientific Computing},
author = {Zhang, F. and Zou, Y. and Chai, S. and Zhang, R. and Cao, Y.},
volume = {93},
number = {1},
year = {2022},
}

@article{SJY:25,
title = {{DGLG}: A Novel Deep Generalized {L}egendre-{G}alerkin Approach to Optimal Filtering Problem},
journal = {IEEE Transactions on Automatic Control},
author = {Shi, J. and Jiao, X. and Yau, S. S.-T.},
volume = {70},
number = {4},
year = {2025},
pages = {2584 - 2590},
}

@book{STW:11,
    author = {Shen, J. and Tang, T. and Wang, L.},
    title = {Spectral Methods: Algorithms, Analysis and Applications},
    publisher = {Springer Berlin, Heidelberg},
    year = {2011}
}

@book{DG:85,
    author = {Devroye, L. and Györfi, L.},
    title = {Nonparametric Density Estimation: The $L_1$ View.},
    publisher = {John Wiley, New York.},
    year = {1985}
}

\section{Appendix}

\begin{proof}[Sketch of the proof of Lemma \ref{L1 Error Bound for Kernel Density Estimation}]
We start from 
\begin{align}\label{eq:decomp}\notag
    &\|p^{N_p,\epsilon_{N_p}} - p\|_{L^1}^2\\
    \leq& 2\|\mathbb{E}p^{N_p,\epsilon_{N_p}} - p\|_{L^1}^2 + 2\|p^{N_p,\epsilon_{N_p}} - \mathbb{E}p^{N_p,\epsilon_{N_p}}\|_{L^1}^2.
\end{align}
%where the first term is the squared $L^1$ of the bias and the second one is that of the variance. We shall estimate them one-by-one.
\textbf{Bias term.}
For a kernel of order $s$, the standard bias expansion (using vanishing moments and Taylor expansion) gives
\begin{equation*}
    \|\mathbb{E}p^{N_p,\epsilon_{N_p}} - p\|_{L^1} \le C_s \epsilon_{N_p}^s \int_{\mathbb{R}} |L(z)|dz (1+o(1)).
\end{equation*}
Squaring and taking expectation on the both sides, we obtain
\begin{equation}\label{eq:bias}
    \mathbb{E}\|\mathbb{E}p^{N_p,\epsilon_{N_p}} - p\|_{L^1}^2 \le C_s^2 \Bigl(\int_{\mathbb{R}} |L(z)|dz\Bigr)^2 \epsilon_{N_p}^{2s}(1+o(1)).
\end{equation}

\textbf{Variance term.}
For any $x$, the variance of the estimator satisfies
\begin{align}\label{eqn-V1}
   & \operatorname{Var}\bigl(p^{N_p,\epsilon_{N_p}}(x)\bigr) \\\notag
    \le &\frac{1}{N_p}\mathbb{E}\bigl[K_{\epsilon_{N_p}}(x-X)^2\bigr]
    = \frac{1}{N_p} \int K_{\epsilon_{N_p}}(x-y)^2 p(y)dy.
\end{align}
Thus,
\begin{align}\label{eqn-V2}\notag
    &\mathbb{E}\|p^{N_p,\epsilon_{N_p}}-\mathbb{E}p^{N_p,\epsilon_{N_p}}\|_{L^1}^2
\le \Bigl( \int \sqrt{\operatorname{Var}(p^{N_p,\epsilon_{N_p}}(x))}\,dx \Bigr)^2\\
    \overset{\eqref{eqn-V1}}\leq& \Bigl( \int \sqrt{ \frac{1}{N_p}\int K_{\epsilon_{N_p}}(x-y)^2 p(y)dy }dx \Bigr)^2\\\notag
   \le &\frac{1}{N_p} \Bigl( \frac{\alpha}{\sqrt{\epsilon_{N_p}}} \int \sqrt{p(y)}dy \Bigr)^2
= \frac{\alpha^2}{N_p\epsilon_{N_p}} \Bigl( \int \sqrt{p(y)}dy \Bigr)^2,
\end{align}
with $\alpha:=\|K\|_{L^2(\R)}^2$, where the first inequality is due to Minkowski's inequality for integrals, i.e.
\begin{align*}  
    &\sqrt{\mathbb{E}\Bigl[\Bigl(\int |p^{N_p,\epsilon_{N_p}}(x)-\mathbb{E}p^{N_p,\epsilon_{N_p}}(x)|dx\Bigr)^2\Bigr]}\\
\le &\int \sqrt{\mathbb{E}[(p^{N_p,\epsilon_{N_p}}(x)-\mathbb{E}p^{N_p,\epsilon_{N_p}}(x))^2]}dx \\
= &\int \sqrt{\operatorname{Var}(p^{N_p,\epsilon_{N_p}}(x))}dx.
\end{align*}
and the third one is from 
\begin{align*}
    &\int \sqrt{ \int K_{\epsilon_{N_p}}(x-y)^2 p(y)dy }dx\\
    \le& \int \sqrt{p(y)} \Bigl( \int K_{\epsilon_{N_p}}(x-y)^2 dx \Bigr)^{1/2} dy\\
    \leq&\frac{\alpha}{\sqrt{\epsilon_{N_p}}} \int \sqrt{p(y)}dy.
\end{align*}
Consequently,
\begin{equation}
    \begin{aligned}\label{eq:variance}
    &\mathbb{E}\|p^{N_p,\epsilon_{N_p}} - \mathbb{E}p^{N_p,\epsilon_{N_p}}\|_{L^1}^2\\
\le& \alpha^2 \Bigl( \int \sqrt{p(z)}dz \Bigr)^2 \frac{1}{N_p\epsilon_{N_p}} (1+o(1)).
\end{aligned}
\end{equation}
The result \eqref{eq:l1-square-bound} follows immediately, by substituting \eqref{eq:bias} and \eqref{eq:variance} back into \eqref{eq:decomp}. Moreover, the optimal rate is immediately obtained by substituting $\epsilon_{N_p^*}$ in \eqref{eq:optimal-bandwidth} into \eqref{eq:l1-square-bound}.
\end{proof}

\addtolength{\textheight}{-12cm} 
\end{document}